\documentclass[a4paper,leqno,11pt]{amsart}

\usepackage{amsfonts,amssymb,verbatim,amsmath,amsthm,latexsym,textcomp,amscd}
\usepackage{latexsym,amsfonts,amssymb,epsfig,verbatim}
\usepackage{amsmath,amsthm,amssymb,latexsym,graphics,textcomp}
\usepackage{mathtools}
\usepackage{graphicx}
\usepackage{color}
\usepackage{url}
\usepackage{enumerate}
\usepackage[mathscr]{euscript}
\usepackage{tikz-cd}
\usetikzlibrary{shapes.geometric}
\usepackage{dsfont}
\usetikzlibrary{matrix}
\usepackage{hyperref}

\input xy

\xyoption{all}

\theoremstyle{definition}
\newtheorem{theorem}{Theorem}[section]
\newtheorem{prop}[theorem]{Proposition}

\newtheorem{cor}[theorem]{Corollary}
\newtheorem{defn}[theorem]{Definition}

\newtheorem{notation}[theorem]{Notation}

\theoremstyle{remark}
\swapnumbers
   \newtheorem{ack}[theorem]{Acknowledgements}
   
 \theoremstyle{definition}  

\newcommand{\C}{{\mathbb C}}

\newcommand{\Z}{{\mathbb{Z}}}
\newcommand{\R}{{\mathbb R}}
\newcommand{\Q}{{\mathbb Q}}

\newcommand\FF{{\mathcal F}}

\newcommand\LL{{\mathcal L}}
\newcommand\MM{{\mathcal M}}

\newcommand\PP{{\mathcal P}}

\newcommand\PMF{{\PP\kern-2pt\MM\FF}}

\newcommand\PML{{\PP\kern-2pt\MM\LL}}

\newcommand{\fsubd}{\mathrel{{\scriptstyle\searrow}\kern-1ex^d\kern0.5ex}}
\newcommand{\bsubd}{\mathrel{{\scriptstyle\swarrow}\kern-1.6ex^d\kern0.8ex}}
\newcommand{\fsubeq}{\mathrel{\raise-.7ex\hbox{$\overset{\searrow}{=}$}}}
\newcommand{\bsubeq}{\mathrel{\raise-.7ex\hbox{$\overset{\swarrow}{=}$}}}

\newcommand{\tsh}[1]{\left\{\kern-.9ex\left\{#1\right\}\kern-.9ex\right\}}

\newcommand{\Index}{\mbox{Index}}
\newcommand{\ind}{\mbox{Ind}}

\newcommand{\coind}{\mbox{Coind}}
\newcommand{\Cindex}{\mbox{C-index}}

\title{Some observations on the index of $C_p$-spaces}

\author{Bikramjit Kundu}
\address{Department of Mathematics, Ramakrishna Mission Vivekananda Educational and Research Institute, Belur Math, Howrah 711202}
\email{bikramju@gmail.com, bikramjit.kundu.math17@gm.rkmvu.ac.in}


\begin{document}

\maketitle

\begin{abstract}
    In this paper, we consider numerical indices associated to spaces with free $C_p$-action. We prove that the Stiefel manifolds provide an example of non-tidy spaces for $p=2$, which are those whose co-index and index disagree. In the case of odd primes, we construct examples of co-index $3$ whose index may be arbitrarily large.
\end{abstract}
\section{Introduction}

In topological combinatorics, one often has to rule out equivariant maps between $G$-spaces. In such cases an important family of invariants comes from various indices associated to $G$-spaces \cite{MT03}. One such index is the Fadell-Husseini index \cite{FaHu88} which is an ideal in the cohomology of $BG$. 

In the case of a cyclic group of prime order the universal space $EG$ may be filtered using spheres. For the group $C_2$, the spheres with antipodal action form the skeleta of $EC_2$. For $p$ odd, the odd dimensional spheres form the odd skeleta of $EC_p$. This leads to the definition of co-index and index of $C_p$-spaces depending of which spheres map to $X$ or which spheres possess a map from $X$. A crucial role here is played by the Borsuk-Ulam theorem which says that there are no $C_p$-maps from spheres of higher dimension to those of lower dimension. 

This paper investigates examples of spaces $X$ whose co-index is $<$ the index. These are called ``non-tidy'' spaces by Matousek \cite{MT03}. For the spheres described, these two are equal. We show using Steenrod operations that Stiefel manifolds provide examples of ``non-tidy'' spaces in many cases (cf. Theorem \ref{Stnt}). The index computations of Stiefel manifolds have been earlier applied to combinatorial problems in \cite{BK21}.

 Matsushita \cite{Ma17} has constructed $C_2$-spaces of co-index $1$ whose index is large. We analogously consider co-index and index for $C_p$-spaces. These have also earlier been studied using Bredon cohomology \cite{BaGh21}, \cite{BaGh2017}. We provide examples of $C_p$-spaces of co-index $3$ and index arbitrarily large (cf. Theorem \ref{hindex}).

\subsection{Organisation} In section \ref{PRL} we recall some basic definitions, useful results regarding index and co-index and some basic notations used in rest of the paper. In section \ref{STK} we construct certain Stiefel manifolds as examples of "non-tidy" space. In section \ref{IND} we construct inductively free $C_p$-spaces whose co-index remain constant but index becomes large.
\begin{ack}
The author would like to thank Samik Basu who had guided every part of this work with uncountable useful discussions. This research was supported by CSIR SRF.
\end{ack}
\section{Preliminaries}\label{PRL}
In this section we will define some basic notations, definitions which will be used frequently in the rest of the paper. The category of the $G$-spaces has objects topological spaces with $G$-action and morphism as $G$-equivariant maps. We recall various indices associated to the free $G$-spaces in the case where $G$ is a cyclic group of prime order. 
\begin{notation}
 We use $S(V)$ to denote the sphere inside a $G$-representation $V$ and $\sigma$ to denote sign representation of $C_2$. For an odd prime $p$, $\lambda$ denotes the one dimensional complex representation of $C_p$ where the chosen generator $\tau$ acts by multiplying a complex number with $e^{\frac{2\pi i}{p}}$. We use the model for the universal space $EC_p$ as the $C_p$-CW-complex whose odd and even dimensional skeletons are defined by \[E^{(2n-1)}C_p=S(n\lambda)\]\[E^{(2n)}C_p=S(n\lambda)\cup_{a} C_p\times D^{2n}\] where $a: C_p\times S^{2n-1}\to S^{2n-1}$ is the action map. If we consider $C_2$, the $k$-th skeleton of $EC_2$ becomes $S((k+1)\sigma)$ which is $S^k$ with anitipodal action.
\end{notation}
Note that for any topological group $G$ there always exists an universal bundle $p : EG \to EG/G=BG$ \cite{Mi56} which is unique upto homotopy. Our construction of universal bundle is slightly specialised which will be useful in proving Theorem \ref{UN}.\\
\begin{defn}
Let $X$ be a free $C_2$-space. Then one may define 
$$ \ind_{C_2}(X) = \text{min} \{ n\geq 0 \mid \exists \mbox{ a } C_2\mbox{-map } X \to S((n+1)\sigma) \}, $$
$$ \coind_{C_2}(X)=\text{max} \{ n\geq 0 \mid \exists \mbox{ a } C_2\mbox{-map }  S((n+1)\sigma) \to X \}.$$
\end{defn}

To rule out the equivariant maps between certain $G$-spaces one useful invariant is Fadell-Husseini index. For a free $G$-space $X$  recall that the homotopy orbit space $X_{hG}:= EG\times_G X$ is homotopically equivalent to $X/G$. Consider the fibration \[X\to X_{hG} \xrightarrow{p_X} BG.\]
\begin{defn}\cite{FaHu88}
The Fadell-Husseini index $\Index_G(X)$ of a $G$-space $X$ is defined as Ker($p^*$) where $p^*: BG\to X_{hG}$.
\end{defn}
Some basic properties of Fadell-Husseini index are \cite{BlPaZi}
\begin{itemize}
    \item \textit{Monotonicity:} If $X\to Y$ is a $G$-equivariant map, then $\Index_G(Y)\subseteq \Index_G(X)$.
    \item\textit{Additivity:} If $(X_1\cup X_2,X_1,X_2)$ is an excisive triple of $G$-spaces, then 
    $$\Index_G(X_1) \Index_G(X_2) \subseteq \Index_G(X_1\cup X_2).$$
    \item \textit{Join:} Let $X$ and $Y$ be $G$-spaces, then $\Index_G(X) \Index_G(Y) \subseteq \Index_G(X*Y)$.
\end{itemize} 
Fadell-Husseini index inspires one more similar numerical invariant as defined below.
\begin{defn}
The cohomological index, denoted by $\Cindex_{G}(X)$ equals the maximum $n$ such that the ideal $\Index_{G}(X)=0$ upto degree $n$.
\end{defn}

We call a $C_2$-space "tidy" if $\ind_{C_2}(X)=\coind_{C_2}(X)$. The existence of $C_2$-maps between "tidy" spaces are completely determined by the $C_2$-index. A $C_2$-space is "non-tidy" if $\ind_{C_2}(X)\neq \coind_{C_2}(X)$. Examples of "non-tidy" spaces are not so trivial. We will construct certain Stiefel manifolds as examples of "non-tidy" spaces in section $3$.
The $C_2$-indices are related by the following inequality, 
\begin{prop}\label{P1}\cite{Ma17}
For any topological space $X$, $\coind_{C_2}(X) \leq \Cindex_{C_2}(X) \leq \ind_{C_2}(X)$. 
\end{prop}
We can generalize the definitions of $C_2$-indices for an arbitrary $G$-space $X$ by replacing the representation sphere to universal spaces $E^{(n)}G$. But we will restrict our attention to $C_p$-spaces in the rest of the section.
\begin{defn}
Let $X$ be a free $C_p$-space. Then one may define 
$$ \ind_{C_p}(X) = \text{min} \{ n\geq 0 \mid \exists \mbox{ a } C_p\mbox{-map } X \to E^{(n)}C_p \}, $$
$$ \coind_{C_p}(X)=\text{max} \{ n\geq 0 \mid \exists \mbox{ a } C_p\mbox{-map } E^{(n)}C_p  \to X \}.$$
\end{defn}
Since there exists $C_p$-map between any two $E^{(n)}C_p$ spaces the definitions are well defined. We can observe similar inequality like free $C_2$-spaces in free $C_p$-space too which is given by \[\coind_{C_p}(X) \leq \Cindex_{C_p}(X) \leq \ind_{C_p}(X).\]
In the case of $C_p$-spaces, the cohomological index is often computed using the height of a cohomology class defined below.
\begin{defn}
The height of a cohomology class is defined by $$ht(v)= \text{min}\{n: v^n=0\}.$$
\end{defn}
\bigskip
Let $V_{l,k}$ denote the Stiefel manifold of $k$ orthonormal vectors in $\R^l$. The group $C_2$ acts on $V_{l,k}$ by sending $(v_1,\ldots,v_k)$ to  $ (\pm v_1,\ldots,\pm v_k)$. The projective Stiefel manifold $X_{l,k}$ is the orbit space $V_{l,k}/C_2$.\\
The mod $2$ cohomology of Stiefel manifold is given by
\[
    H^*(V_{l,k})=\Lambda_{\Z_2}(x_{l-k},\ldots,x_{l-1}).
\]
From the Serre spectral sequence of the fibration $V_{l,k}\xrightarrow{i} X_{l,k} \xrightarrow{p} BC_2$ we obtain the mod $2$ cohomology of $X_{l,k}$ additively (cf. Theorem 1.6 \cite{GiHa68}).
\begin{theorem}\label{CP}
   $ H^*(X_{l,k};\Z_2)=\Z_2[z]/\langle z^N \rangle \otimes \Lambda_{\Z_2}(z_{l-k},\ldots,z_{l-1})$, where degree of $z$ is 1 and $N=$ min $\{j : l-k+1\leq j\leq l\}$ and $\binom{l}{j} \neq 0$ (mod 2). Moreover $p^*(u)=z$ where $u$ is the generator of $H^1(BC_2;\Z_2)$ and $i^*(z_i)=x_i$.
\end{theorem}

\begin{theorem}\label{IV}
$\Index_{C_2}V_{l,k}$ is the ideal $\langle u^N\rangle$ in the cohomology of $BC_2$, where $N$ is as described in \eqref{CP}.
\end{theorem}
\begin{proof}
Considering the fibration 
\[
V_{l,k} \rightarrow X_{l,k} \rightarrow BC_2 
\]
the proof will directly follow from the Theorem $2.8$.
\end{proof}
\section{Stiefel manifolds as an example of non-tidy spaces}\label{STK}
In this section we will try to construct certain Stiefel manifolds as an example of "non-tidy" spaces. Note that as $V_{l,k}$ is $l-k-1$ connected, by equivariant obstruction theory there is a $C_2$-map from $S((l-k)\sigma) \to V_{l,k}$. Therefore we can say that the co-index of $V_{l,k}$ is at least $l-k-1$. By monotonicity of Fadell-Husseini index we can rule out the existence of $C_2$-equivariant map from $V_{l,k}\to S(r\sigma)$ if $r< N$, where  $N=$ min $\{j : l-k+1\leq j\leq l\}$ and $\binom{l}{j} \equiv 1 \pmod{ 2}$. The next result addresses the following question. Does there exists a $C_2$-map \[f: V_{l,k} \to S(N\sigma)\] for suitable $l$ and $k$?
\begin{theorem}\label{Stnt}
For $l=k-1+\alpha2^s$, $k<2^s$ we have 
$\Cindex_{C_2}(V_{l,k})= \alpha 2^{s}-1$. Further if $s$ is the least positive integer such that $k<2^s$ then $\ind_{C_2}(V_{k-1+\alpha 2^{s},k}) > \alpha 2^{s}-1$.
\end{theorem}

\begin{proof}
From Theorem \eqref{IV} it follows that $\Cindex_{C_2}(V_{l,k})$ is $N-1$ where $N$ is as described in \eqref{CP}. We have
\begin{align*}\label{BN}
\binom{l}{l-k+1}   &=\frac{(k-1+2^s\alpha)(k-2+2^s\alpha)\cdots (1+2^s\alpha)}{(k-1)\cdots1}\\ &=\frac{(k-1+2^s\alpha)}{(k-1)}\cdot \frac{(k-2+2^s\alpha)}{(k-2)}\cdots \frac{(1+2^s\alpha)}{1}.
\end{align*}

Now if $k-i$ is odd the expression $\frac{(k-i+2^s\alpha)}{(k-i)}$ is odd. If $k-i$ is even we can factor out $2^m$ part ($m<s$) from both the numerator and denominator and the expression becomes odd. Therefore we can conclude \begin{equation}\label{BN}
\binom{l}{l-k+1} \equiv 1 \pmod 2.
\end{equation}
Hence $N=l-k+1=\alpha 2^{s}$. This completes the first part of the proof.\\ 
For computing the topological index suppose there exists a $C_2$-map $f: V_{l,k} \to S(N\sigma) $, then it will induce the following commutative diagram between fibrations
\[
\xymatrix{ V_{l,k} \ar[rr]^f \ar[d] & & S(N\sigma) \ar[d] \\ 
          X_{l,k} \ar[rr]^-{f_{hC_2}} \ar[d] && \R P^{N-1} \ar[d]\\ 
           BC_2 \ar@{=}[rr]    &  & BC_2. }
           \]
As both $\Index_{C_2}V_{l,k}$ and $\Index_{C_2}S(N\sigma)$ is $\langle u^N\rangle$  \[ f^*(\epsilon_{N-1})=
\begin{cases}
x_{N-1} & \text{(mod $I^2$)\quad if $2k>l$}\\
x_{N-1} & \text{if $2k<l$}
\end{cases}
\] where $\epsilon_{N-1}$ is the generator of top cohomology of $S(N\sigma)$ and $I$ is the ideal $\langle x_{l-k},\cdots, x_{l-1} \rangle$.
Observe that for $k>0$ 
\begin{align*}
Sq^k(x_{N-1}) &=
Sq^kf^*(\epsilon_{N-1})\\
&=f^*Sq^k\epsilon_{N-1}\\
&=0.
\end{align*}
Now if we can show that $Sq^k(x_{N-1})\neq 0$ for some $k>0$, we will obtain a contradiction.\\

There is a map $i: \R P ^{l-1} \to SO(l)/SO(l-k)\cong V_{l,k}$, (Ch.(5)\cite{MoTa68}). Consider the diagram 
\[
\xymatrix{
H^{N-1}(V_{l,k})\ar[rr]^{i^*} \ar[d]^{Sq^{2^{s-1}}} && H^{N-1}(\R P^{l-1})\ar[d]^{Sq^{2^{s-1}}}\\
H^{N-1+2^{s-1}}(V_{l,k})\ar[rr]^{i^*} && H^{N-1+2^{s-1}}(\R P^{l-1}).
}
\]
We know from the property of Steenrod squares that (Ch.4, \cite{AH02})
\[
Sq^{2^{s-1}}(u^{N-1})=\binom{N-1}{2^{s-1}}u^{N-1+2^{s-1}}.
\]
The expression would be non zero for following two conditions 
\begin{equation}\label{C1}
\binom{N-1}{2^{s-1}} \equiv 1 \pmod{2}  \end{equation} and \[N-1+2^{s-1}\leq l-1,\] i.e.
\begin{equation}\label{C2}
    N+2^{s-1} \leq l.
\end{equation}
We have \[l=k-1+\alpha 2^{s},\]\[N=\alpha 2^{s}.\]
Expanding \eqref{C1} we get
\[
\frac{(N-1)(N-2)\cdots(N-2^{s-1})}{2^{s-1}\cdots 1} 
=\frac{(2^s\alpha-1)}{1}\cdots \frac{(2^s\alpha-2^{s-1}+1)}{(2^{s-1}-1)}\cdot\frac{2^s\alpha-2^{s-1}}{2^{s-1}}
\] which is odd by a similar argument as in \eqref{BN}.

So for $\alpha\geq 1$ and  $l=k-1+\alpha 2^{s}$  both the conditions \eqref{C1} and \eqref{C2} are satisfied implying, $Sq^{2^{s-1}}(u^{N-1})\neq 0$ and we obtain a contradiction. This implies $\ind_{C_2}(V_{{k-1+\alpha2^{s}},k})>\alpha 2^{s}-1$.
\end{proof}
From Proposition \ref{P1} we obtain the following result.
\begin{cor}
 With $k$, $s$ and $\alpha$ as above, $V_{{k-1+\alpha2^{s}},k}$ are examples of "non-tidy" spaces.
\end{cor}

\section{$C_p$-space of high index}\label{IND}
In this section we provide examples of spaces with small co-index and high C-index. We first prove a lemma which identifies the space of co-index $1$.\\

Let $X$ be a free $C_p$ path connected space and $\bar{X}$ be its orbit space. From covering space theory we have $\pi_1(\bar{X})/\pi_1(X)\cong C_p$. Let $f : X\to Y$  be a $C_p$-equivariant map between two free path connected $C_p$-spaces. This map will induce $\bar{f_*} : \pi_1(\bar{X})/\pi_1(X) \to \pi_1(\bar{Y})/\pi_1(Y)$. We have a commutative diagram 
\begin{equation}\label{C4}
\xymatrix{
\pi_1(\bar{X}) \ar[rr]^{\bar{f_*}} \ar[d] && \pi_1(\bar{Y}) \ar[d]\\
C_p \ar@{=}[rr] && C_p.\\
}\end{equation}
We call an element $\alpha \in \pi_1(\bar{X})$ prime to $p$ if it does not belong to $\pi_1(X)$.
\begin{theorem}\label{UN}
For an 1-connected free $C_p$-space $X$, there exists a map $g: E^{(2)}C_p \to X$ iff $\pi_1(\bar{X})$ has an element prime to $p$ whose order is $p$.
\end{theorem}
\begin{proof}
The proof is similar to Theorem $2.2$ described in \cite{Ma17}. Let $\tau \in \pi_1(\overline{E^{(2)}C_p)}\cong C_p$ be the generator of $C_p$. By the commutative diagram \eqref{C4}, existence of $g$ implies $\bar{g_*}(\tau)$ is an element prime to $p$ of order $p$.\\
Let $\beta\in\pi_1(\bar{X})$ be an element prime to $p$ such that $\beta^p=1$. Let $\bar{\gamma}$ be a representative of $\beta$ and $\gamma$ be it's lift in $X$. We can choose $\gamma$ as $$\gamma:[0,2\pi/p]\to X.$$ Since $\beta$ is prime to $p$ we have a non-trivial generator $\tau$ of $C_p$, such that $\tau\gamma(0)=\gamma(\frac{2\pi}{p})$. Define $\phi=\gamma\cdot(\tau\gamma)\cdots(\tau^{p-1}\gamma)$. Then $\phi: S^1=E^{(1)}C_p\to X$ is a $C_p$-equivariant map and \[p_*[\phi]=\beta^p=1\] where $p:X\to \bar{X}$ is the covering projection. Since $p_*$ is injective in $\pi_1$ we can conclude $\phi$ is null-homotopic and it can be extended to a map $g:E^{(2)}C_p\to X$ by equivariant obstruction theory \cite{MT03}.
\end{proof}
\bigskip
We want to construct a sequence of spaces $X(k)$ whose $\Cindex_{C_p}(X(k))$ becomes large and $\coind_{C_p}(X(k))$ is small. Start with $X(0)=S^3$ where $C_p$ acts freely and the left action is generated by $g\cdot(z_0,z_1)=(e^{2\pi i/p}z_0,z_1)$ and the right action is generated by $(z_o,z_1)\cdot g= (e^{2\pi i/p}z_0,e^{2\pi i/p}z_1)$. Define $X(k+1)=X(k)\times_{C_p}S^3$. We see \begin{equation}\label{C5}
S^3\to X(k+1)\to X(k)/C_p
\end{equation} is a $S^3$ bundle over $X(k)/C_p$. More specifically, it is the bundle $S(L\oplus \epsilon_{\C})$ over $X(k)/C_p$, where the line bundle $L$ is classified by \[X(k)/C_p \xrightarrow{\psi_{k}} BC_p \xrightarrow{\phi} BS^1.\] Recall that \[H^*(BC_p;\Z_p)=\Z/p[\epsilon,y]/(\epsilon^2)\] where $|y|=2$, $|\epsilon|=1$ and the two generators are related by Bockstein, $\beta(\epsilon)=y$. Also recall, \[H^*(BS^1;\Z/p)\equiv H^*(\C P^{\infty};\Z/p)\equiv \Z/p[x]\] where $\phi^*(x)=y$. \\
From the fiber bundle \eqref{C5} we have a $C_p$-equivariant inclusion $S^3 \to X(k+1)$. This implies $\coind_{C_p}(X(k+1)) \geq 3$.
\begin{theorem}\label{hindex}
$\Cindex_{C_p}(X(k+1)) \geq 2k+2$ and $\coind_{C_p}(X(k+1))=3$.
\end{theorem}
\begin{proof}
 Consider the fibration 
\begin{equation}\label{C6}
S^3/C_p=L_p(3) \xrightarrow{i} X(k+1)/C_{p} \xrightarrow{} X(k)/C_p 
\end{equation} and the commutative diagram of fibrations \[
\xymatrix{S^3 \ar[rr] \ar[d] && X(k+1) \ar[rr] \ar[d] && EC_p \ar[d]\\
L_p(3) \ar[rr]^{i} && X(k+1)/C_{p} \ar[rr]^{\psi_{k+1}} && BC_p.
}\]
As $\coind_{C_p}(X(k+1))\geq 3$, $\Cindex_{C_p}(X(k+1))\geq 3$. This implies $\Index_{C_p}(X(k+1))$ does not have any element in degree 3 and $$\psi_{k+1}^*(y)\neq 0, \psi_{k+1}^*(\epsilon)\neq0, \psi_{k+1}^*(\epsilon y)\neq 0$$ all of which get pulled back to the respective generators of the cohomology of $L_p(3)$. The spectral sequence for \eqref{C6} collapses at $E_2$ page.

 Consider the commutative diagram of fibrations
\begin{equation}\label{CD3}
\xymatrix{X(k+1) \ar[rr] \ar[d] && X(k+1)/C_{p} \ar[rr]^{\Lambda_{k+1}} \ar[d] && P(L\oplus \epsilon) \ar[d]\\
X(k)/C_{p} \ar@{=}[rr] && X(k)/C_{p} \ar@{=}[rr] && X(k)/C_{p} \quad .}
\end{equation}
From the construction of Chern classes \cite{BT82} \[ H^*(P(L\oplus \epsilon); \Z/p)= H^*(X(k)/C_{p})[x]/(x^2+z_kx)\] where $z_k$ is the first Chern class of the line bundle $L$ over $X(k)/C_p$ as described above, that is \[c_1(L\oplus \epsilon)=c_1(L)=\psi_k^*(y)=z_k.\]Let \[\psi^*_{k+1}(\epsilon)=e_{k+1},\quad \psi^*_{k+1}(y)=z_{k+1}.\] As cohomology of $L_p(3)$ is freely generated by $i^*(e_{k+1})$, $i^*(z_{k+1})$, $i^*(e_{k+1}z_{k+1})$,  Lerray-Hirsch theorem for the bundle\[ L_p(3)\xrightarrow{i} X(k+1)/C_p \to X(k)/C_p\] implies $H^*(X(k+1)/C_p;\Z/p)$ is the $H^*(X(k)/C_p;\Z/p)$-module generated by $e_{k+1}$,$z_{k+1}$,$e_{k+1}z_{k+1}$. To find the ring structure of $H^*(X(k+1)/C_p;\Z/p)$ first observe that $\Lambda^*$ is a ring map. From the right hand square of the diagram \eqref{CD3} we have \[\Lambda^*(x)=\psi^*_{k+1}(y)=z_{k+1}, \]\[\Lambda^*(x^2+z_kx)=0\]\[\Rightarrow z_{k+1}^2+z_kz_{k+1}=0.\] And $\epsilon^2=0$ implies $e_{k+1}^2=0$. Thus \[H^*(X(k+1)/C_p;\Z_p)=H^*(X(k)/C_p;\Z_p)[e_{k+1},z_{k+1}]/(e^2_{k+1},z^2_{k+1}+z_kz_{k+1}).\] Therefore in this ring,
\begin{align*}z_{k+1}^n &= \pm z^{n-1}_{k+1}z_k\\
              &=\pm z_{k+1}z_k^{n-1}.
\end{align*}
This gives $ht(z_{k+1})=ht(z_{k})+1$. As $\beta(e_{k+1})=z_{k+1}$, we conclude that
\begin{align*}
\Cindex_{C_p}(X(k+1)) =&
2(ht(z_{k+1})-1)\\
 &\text{or}\\
=&2ht(z_{k+1})-1.
\end{align*}
So we have $\Cindex_{C_p}(X(k))=2k+2$ or $2k+3$, which is arbitrary large as $k$ increases. This completes the proof of the first part of the theorem.\\

For the second part consider any arbitrary $C_p$-map $f:S^3\to X(k+1)$. We will show this map is not null-homotopic. This implies that the map does not extend to $E^{(4)}C_p$ and thus $\coind_{C_p}(X(k+1))=3$. By construction of $X(k)$ consider the commutative diagram 
\begin{equation}\label{CD4}
\xymatrix{ S^3 \ar[r]^{f} \ar[d]^{q}          & X(k+1) \ar[d] \ar[rdd]\\
S^3/C_p \ar[r] \ar[rrd]^{g}    & X(k+1)/C_p \ar@{-->}[rd]\\
                            &                         & X(k)/C_p.}
\end{equation}

If for every $r$, image of $\pi_1(S^3/{C_p})$ is non-zero under the composite map \[S^3/{C_p}\to X(k+1)/{C_p} \to \cdots \to X(r)/{C_p}\] we reach $X(0)/{C_p}$ and have $g_0: S^3/C_p\to X(0)/C_p=S^3/C_p$ for which $g_{0*}$ is not 0. That means $g_{0*}$ induces isomorphism in $\pi_1$. This implies $g_0$ lifts to an $C_p$-equivariant map $\tilde{g}_0: S^3\to S^3$ with 
\[\tilde{g}_{0*}: H_3(S^3;\Q) \to H_3(S^3;\Q)\] is an isomorphism. The quotient map $q: S^3\to S^3/C_p$ induces an isomorphism \[q_*: H_3(S^3;\Q)\cong H_3(S^3/C_p;\Q).\]
Applying $H_3(-;\Q)$ to the diagram \eqref{CD4} gives 
\[
\xymatrix{ H_3(S^3;\Q) \ar[r]^{f_*} \ar[d]^{q_*}          & H_3(X(k+1);\Q) \ar[d] \ar[rdd]\\
H_3(S^3/C_p;\Q) \ar[r] \ar[rrd]^{g_0* }    & H_3(X(k+1)/C_p;\Q) \ar@{-->}[rd]\\
                            &                         & H_3(X(0)/C_p;\Q).}
\]
In above diagram as the composite $g_*q_*\neq 0$ we must have \[f: S^3\to X(k+1)\] not homotopic to zero as it induces non-trivial map on rational homology.\\

Now we may assume for $k=r$ \[
\xymatrix{
S^3/C_p \ar[r] \ar[rd]^{g} & X(r+1)/C_p\ar[d] \\ 
                            & X(r)/C_p}
\]
and hence the map $g_*$ is 0 in $\pi_1$. From the covering space theory we have a lift of $g$ in the universal cover of $X(k)/C_p$ by the following diagram
\[
\xymatrix{             
                        & S^3\times S^3 \times\cdots S^3 \ar[d] \\
                        & X(k) \ar[d]\\
S^3/C_p \ar[r]^g \ar@{-->}[ruu]^{\lambda}  &   X(k)/C_p .}\]
The 3-equivalence $S^3\to K(\Z/3)$ implies the following isomorphism between the based homotopy classes of maps
 \begin{align*}[S^3/C_p, S^3\times S^3\cdots S^3] &\cong [S^3/C_p,K(\Z/3)\times \cdots K(\Z/3)]\\
                                                 & \cong \bigoplus_1^{k+1} H^3(S^3/C_p;\Z)\\
                                                 & \cong \bigoplus_1^{k+1}  \Z.
\end{align*}

If $\lambda$ is not homotopic to $*$,\[\lambda^*: H^3((S^3)^{k+1};\Z) \to H^3(S^3/C_p;\Z) \] is non-trivial. This implies \[ \lambda^*: H^3((S^3)^{k+1};\Q) \to H^3(S^3;\Q) \] is non-zero. From the long exact sequence of fibration $S^3\times\cdots S^3 \to X(k)/C_p$ we have $H^3(X(k)/C_p;\Q) \cong H^3((S^3)^{k+1};\Q) $.
This implies that the map $ gq: S^3\to X(k)/C_p $ in the commutative diagram \eqref{CD4} is not homotopic to zero. Thus $f: S^3\to X(k+1)$ is not null-homotopic as it induces an isomorphism on $H^*(-;\Q)$.\\
If $\lambda$ is homotopic to $*$ then $g$ is null-homotopic. From the covering lifting property of fibration we have
\[
\xymatrix{
                     && S^3/C_p\ar[d]^{i} \\
                     && X(k+1)/C_p\ar[d] \\
    S^3/C_p  \times I  \ar@{-->}[rru] \ar[r]  & S^3/C_p \times \{1\} \ar[r] \ar@{-->}[ruu]^{\mu} & X(k)/C_p.
   }
\]
The map $S^3/C_p\to X(k+1)/C_p$ on orbit spaces induced from $f$ is non-trivial, so we deduce that
$\mu$ is an isomorphism in $\pi_1$ and in $\Z/p$-cohomology. Now consider the commutative diagram
\begin{equation}\label{CD6}
\xymatrix{
H^3(X(k+1)/C_p;\Z) \ar[r]^{i^*} \ar[d] & H^3(S^3/C_p;\Z) \ar[r]^{\mu^*} \ar[d] & H^3(S^3/C_p;\Z) \ar[d]\\
H^3(X(k+1)/C_p;\Z/p) \ar[r] & H^3(S^3/C_p;\Z/p) \ar[r]^{\mu*}_{\cong} & H^3(S^3/C_p;\Z/p). 
}
\end{equation}
We have $\mu^*$ non-zero in  upper row. We will show $i^*$ is non-zero in upper row. It suffices to prove this rationally by the commutative diagram \[ \xymatrix{
H^3(X(k+1);\Q)\ar[r]^{i^*} & H^3(S^3;\Q)\\
H^3(X(k+1)/C_p;\Q) \ar[r]\ar[u]^{\cong} & H^3(S^3/C_p;\Q)\ar[u]^{\cong}.
}\] We will show $i^*:H^3(X(k+1);\Q) \to H^3(S^3;\Q)$ is non-zero. For this we will analyze Serre spectral sequence associated to the principal fibration \[S^3\to X(k+1)/C_p \to X(k)/C_p.\] As mentioned earlier in this section this is exactly the $S(L\oplus \epsilon_{\C})$ bundle over $X(k)/C_p$. Observe that $\pi_1(X(k)/C_p)=\oplus_{k+1}\Z/p$ does not act non-trivially on Aut$(\Z)$. The  $\Z$-cohomology spectral sequence associated to the bundle collapses at $E_2$ page. Denote by $\beta$ the generator of top cohomology of $S^3$. Then the differential is given by \[ d(\beta)=e(L\oplus\epsilon_{\C})=0,\] where $e$ denotes the Euler class of the bundle. This implies $i^*:H^3(X(k+1);\Z) \to H^3(S^3;\Z)$ is non-trivial and so with rational coefficient. This implies that the homomorphism  $\mu^*\circ i^*$ in the top row of the diagram \eqref{CD6} is non-zero. Therefore the top horizontal map is non-trivial in the following commutative diagram as the down horizontal map is non-zero.\[\xymatrix{
H^3(X(k+1);\Q)\ar[r]^{f^*} & H^3(S^3;\Q)\\
H^3(X(k+1)/C_p;\Q) \ar[r]^{\mu^*\circ i^*}\ar[u]^{\cong} & H^3(S^3/C_p;\Q)\ar[u]^{\cong}
.}\] This implies $f:S^3\to X(k+1)$ is not null-homotopic and completes the proof.

\end{proof}

\end{document}